\documentclass[11pt,a4paper]{amsart}
\usepackage{amsfonts,amssymb,amsmath,amsthm}
\usepackage{stmaryrd} 
\usepackage{tikz}
\usepackage[text={15cm,23cm}]{geometry} 
\usepackage[colorlinks=true,%
            linkcolor=red!50!black,%
            citecolor=blue!50!black,%
            urlcolor=darkgray]{hyperref}  
\theoremstyle{theorem}
\newtheorem{theorem}{Theorem}[section]
\newtheorem{lemma}[theorem]{Lemma}

\newtheorem{corollary}[theorem]{Corollary}
\newtheorem{proposition}[theorem]{Proposition}
\newtheorem{remark}[theorem]{Remark}

\DeclareMathOperator{\Tr}{Tr}
\newcommand{\R}{{\mathbb R}}
\newcommand{\N}{{\mathbb N}}
\newcommand{\Q}{{\mathbb Q}}
\newcommand{\Z}{{\mathbb Z}}

\newcommand{\coloneq}{\mathbin{\hbox{\raise0.08ex\hbox{\rm :}}\!\!=}}
\newcommand{\eqcolon}{\mathbin{=\!\!\hbox{\raise0.08ex\hbox{\rm :}}}}

\newcommand{\rme}{{\mathrm e}}
\newcommand{\rmi}{{\mathrm i}}

\theoremstyle{definition}
\newtheorem{example}[theorem]{Example}
\newtheorem{definition}[theorem]{Definition}

\title{Quantum Kronecker fractions}
\author{S.J. Evans }
\address{Department of Mathematical Sciences,
Loughborough University, Loughborough LE11 3TU, UK}
\email{S.J.Evans@lboro.ac.uk}

\author{A.P. Veselov}
\address{Department of Mathematical Sciences,
Loughborough University, Loughborough LE11 3TU, UK}
\email{A.P.Veselov@lboro.ac.uk}

\author{B. Winn}
\address{Department of Mathematical Sciences,
Loughborough University, Loughborough LE11 3TU, UK}
\email{B.Winn@lboro.ac.uk}

\begin{document}
\maketitle

\begin{abstract}
A few years ago Morier-Genoud and Ovsienko introduced an interesting quantization of the real numbers as certain power series in a quantization parameter $q.$
It is known now that the golden ratio has minimal radius among all these series.
We study the rational numbers having maximal radius of convergence equal to 1, which we call Kronecker fractions.
We prove that the corresponding continued fraction expansions must be palindromic and describe all Kronecker fractions with prime denominators. We found several infinite families of Kronecker fractions and all Kronecker fractions with denominator less than 5000. We also comment on the irrational case and on the relation with braids, rational knots and links.
\end{abstract}

\section{Introduction}

Recently Morier-Genoud and Ovsienko \cite{MGO1, MGO2} introduced interesting quantum versions of real numbers as formal power series in a quantization parameter $q$.
For a rational number $x=\frac{r}{s}$ these series converges to a rational function
$\left[ \frac{r}s \right]_q = \frac{\mathcal{R}(q)}{\mathcal{S}(q)},$
where $\mathcal{R}$, $\mathcal{S}$ are coprime, monic polynomials with 
non-negative integer coefficients (see the next section for more details).
For natural number $n$ the quantization $[n]_q$ coincides with Euler's $q$-integer
$[n]_q = 1 + q +\dots + q^{n-1}.$

The question about convergence of the corresponding power series for all real $x$ was studied in \cite{LMGOV}, where it was proved that for any positive $x$ the radius of convergence $R(x)$ of the series~$\left[x\right]_q$
satisfies the inequality
$
R(x) > 3-2\sqrt{2}
$
and conjectured that the optimal lower bound for $R(x)$ is $R_*=\frac{3-\sqrt{5}}{2}$ holding only for the golden ratio and its equivalents.
The proof of this conjecture has recently been announced in \cite{EGMS}.

In this paper we study the rational numbers with maximal radius of convergence, which for non-integers is equal to 1.
The corresponding numbers we call {\it Kronecker fractions.} The reason is that the denominator $\mathcal{S}(q)$ of such fractions must be monic, has integer coefficients and all roots lying in the unit disc. Kronecker \cite{K} proved that such polynomials must be a product of cyclotomic polynomials, and/or a power of $q$.

The simplest examples of Kronecker fractions are $x=\frac{1}{n}$ with $$\left[ \frac{1}n \right]_q=\frac{q^{n-1}}{q^{n-1} +\cdots + q + 1}=\frac{q^{n-1}}{[n]_q},\qquad n\in\N,$$
but there are more intriguing families of examples like $[x]_q=\left[ \frac{n^2 + n -1}{n(n + 1)(n + 2)} \right]_q$ with the denominator $\mathcal{S}(q)=[n]_q[n + 1]_q[n + 2]_q(1 - q + q^2).$
In particular, for $n=4$ we have $$\left[ \frac{19}{120} \right]_q = \frac{q^6(1+2q+3q^2+3q^3+3q^4+3q^5+2q^6+q^7+q^8)}{[4]_q[5]_q[6]_q(q^2- q + 1)}.$$

We have also sporadic examples (not belonging to any known family) such as
$$\left[ \frac{49}{160} \right]_q = \frac{q^3(q^5 + q^4 + q^3 + 2q^2 + q + 1)(q^4 + q^3 + 2q^2 + 2q + 1)}{[4]_q^2[5]_q(q^2+1)}.$$

We prove that Kronecker fractions must have palindromic continued fractions and show that the Kronecker fractions with prime denominators must have the form $\frac{1}{p}$ or $\frac{1-p}{p}$.
We found all the Kronecker fractions with the denominator less than 5000 and several infinite families (see the Appendix for all cases known so far), but a complete description is still to be found. We finish with the discussion of the irrational case and of the relation with braids, rational knots and links.


\section{Quantum numbers}
\subsection{Quantum rationals}

    For a rational number $\frac{r}{s}$, we consider its \textit{regular 
continued fraction} by
    \begin{equation*}
        \frac{r}{s} = [a_1, a_2, a_3, \ldots, a_N] \coloneq a_1 + \cfrac{1}{a_2 + \cfrac{1}{a_3 + \cfrac{1}{\ddots + \cfrac{1}{a_N}}}},
    \end{equation*}
    where $a_1$ is an integer and $a_i$ is a positive integer for all $i \geq 2$. Since
$
  \frac{1}{a_N} = \frac{1}{a_N-1 + \frac{1}{1}}
$
we can always assume $\frac{r}{s} = [a_1,a_2,\ldots,a_{2m}]$
(i.e.~the length of the list of partial quotients is even).

\begin{definition} (Morier-Genoud, Ovsienko \cite{MGO1}) \label{def:1}
    Given a continued fraction $[a_1, a_2, \dots, a_{2m}]$, define 
its \textit{$q$-deformation} by 
    \begin{equation*}
        [a_1, a_2, \dots, a_{2m}]_q \coloneq [a_1]_q + \cfrac{q^{a_1}}{[a_2]_{q^{-1}} + \cfrac{q^{-a_2}}{[a_3]_q + \cfrac{q^{a_3}}{[a_4]_{q^{-1}} + \cfrac{q^{-a_4}}{\ddots + \cfrac{q^{a_{2m-1}}}{[a_{2m}]_{q^{-1}}}}}}}
    \end{equation*}
\end{definition}

\begin{example}
  $\frac38=[0,2,1,2]$, so
$$
    \left[ \frac38 \right]_q = \cfrac{1}{[2]_{q^{-1}} + \cfrac{q^{-2}}{[1]_q
+\cfrac{q}{[2]_{q^{-1}}}}}  
=\frac{q^{4} + q^{3} + q^{2}}{q^{4} + 2q^{3} + 2q^{2} + 2q + 1}.
$$
\end{example}

Clearly a quantum rational is a rational function of $q$:
\begin{equation*}
  \left[ \frac{r}s \right]_q = \frac{\mathcal{R}(q)}{\mathcal{S}(q)}.
\end{equation*}
In \cite{MGO1} it is further shown that:
\begin{itemize}
\item $\mathcal{R}(1)=r$, $\mathcal{S}(1)=s$,
\item $\mathcal{S}(0)=1$,
\item $\mathcal{R}$, $\mathcal{S}$ are coprime, monic polynomials with 
non-negative integer coefficients
\item $\deg(\mathcal{R}) = a_1 + a_2 + \cdots + a_{2m} - 1$ and
$
    \deg(\mathcal{S}) = \deg(\mathcal{R}) - a_1.
$
\end{itemize}
Note that $\mathcal{R}$ and $\mathcal{S}$ depend on both $r$ and $s$.
For example, the denominator ``$5$'' is quantised differently in
$\frac15$ and $\frac25$:
\begin{equation*}
  \left[ \frac15 \right]_q = \frac{q^4}{q^4+q^3+q^2+q+1},  \quad \left[ \frac25 \right]_q   =\frac{q^{3} + q^{2}}{q^{3} + 2q^{2} + q + 1}.
\end{equation*}

One can visualize these quantum rationals on the Conway topograph \cite{Conway} using the quantized local rules shown in Fig.~1.
\medskip

\begin{figure}[h]
\begin{center}
	\begin{tikzpicture}
	\node at (-1, 0.75) {$\frac{\mathcal{R}}{\mathcal{S}}$};
	\node at (1, 0.75) {$\frac{\mathcal{R'}}{\mathcal{S'}}$};
	\node at (0, 2.5) {$\frac{\mathcal{R}+q^\ell\mathcal{R'}}{\mathcal{S}+q^\ell\mathcal{S'}}$};
	
	\draw[thick] (0, 0) -- node[right, text=red] {$q^\ell$} (0, 1.5);
	\draw[thick] (0, 1.5) -- (1, 2.5);
	\draw[thick] (0, 1.5) -- (-1, 2.5);
	\end{tikzpicture}
	\quad \quad \quad
	\begin{tikzpicture}
	\node at (-1, 0.75) {$\frac{\mathcal{R}}{\mathcal{S}}$};
	\node at (1, 0.75) {$\frac{\mathcal{R'}}{\mathcal{S'}}$};
	
	\draw[thick] (0, 0) -- node[right, text=red] {$q^\ell$} (0, 1.5);
	\draw[thick] (0, 1.5) -- node[below, right, text=red] {$q^{\ell+1}$} (1, 2.5);
	\draw[thick] (0, 1.5) -- node[above, right, text=red] {$q$} (-1, 2.5);
	\end{tikzpicture}
  \end{center}
\medskip
  \begin{center}
	\begin{tikzpicture} 
    \node at (1.75,0.75) {$\frac{1}{0}$};
	\node at (-1.75,0.75){$\frac{0}{1}$};
	\node at (0,4) {$\frac{1}{1}$};
	\node at (3.11, 3.8) {$\frac{1+q}{1}$};
	\node at (2.30, 5.76) {$\frac{1+q+q^2}{1+q}$}; 
	\node at (4.37, 1.98) {$\frac{1+q+q^2}{1}$}; 
	\node at (0.87, 6.97) {$\frac{1+q+q^2+q^3}{1+q+q^2}$}; 
	\node at (4.38, 5.85) {$\frac{1+q+2q^2+q^3}{1+q+q^2}$}; 
	\node at (5.81, 3.35) {$\frac{1+2q+q^2+q^3}{1+q}$}; 
	\node at (4.43, 0.34) {$\frac{1+q+q^2+q^3}{1}$}; 
	\node at (-3.11, 3.8) {$\frac{q}{1+q}$};
	\node at (-2.10, 5.76) {$\frac{q+q^2}{1+q+q^2}$}; 
	\node at (-4.37, 2.18) {$\frac{q^2}{1+q+q^2}$}; 
	\node at (-0.87, 6.97) {$\frac{q+q^2+q^3}{1+q+q^2+q^3}$}; 
	\node at (-4.38, 5.85) {$\frac{q+q^2+q^3}{1+2q+q^2+q^3}$}; 
	\node at (-5.81, 3.35) {$\frac{q^2+q^3}{1+q+2q^2+q^3}$}; 
	\node at (-4.43, 0.34) {$\frac{q^3}{1+q+q^2+q^3}$}; 
	
	\draw[thick] (0, -0.5) -- node[right, text=red] {\small $1$} (0,2);
	
		\draw[thick] (0,2) -- node[below, text=red] {\small $q$} (1.73,3);
			
			\draw[thick] (1.73, 3) -- node[right, text=red] {\small $q$} (1.99, 4.48);	
	
				\draw[thick] (1.99, 4.48) -- node[above, text=red] {\small $q$} (1.31, 5.53);
				
					\draw[thick] (1.31, 5.53) -- node[above, text=red] {\small $q$} (0.36, 5.87);
					\draw[thick] (1.31, 5.53) -- node[right, text=red] {\small $q^2$} (1.23, 6.53);
					
				\draw[thick] (1.99, 4.48) -- node[below, text=red] {\small $q^2$} (2.98, 5.24);	
				
					\draw[thick] (2.98, 5.24) -- node[right, text=red] {\small $q$} (3.24, 6.21);
					\draw[thick] (2.98, 5.24) -- node[below, text=red] {\small $q^3$} (3.98, 5.24);
	
			\draw[thick] (1.73, 3) -- node[below, text=red] {\small $q^2$} (3.14, 2.49);
	
				\draw[thick] (3.14, 2.49) -- node[below, text=red] {\small $q$} (4.29, 2.97);
	
					\draw[thick] (4.29, 2.97) -- node[right, text=red] {\small $q$} (4.79, 3.84);
					\draw[thick] (4.29, 2.97) -- node[below, text=red] {\small $q^2$} (5.26, 2.71);
				
				\draw[thick] (3.14, 2.49) -- node[left, text=red] {\small $q^3$} (3.72, 1.38);
					
					\draw[thick] (3.72, 1.38) -- node[below, text=red] {\small $q$} (4.63, 0.96);
					\draw[thick] (3.72, 1.38) -- node[left, text=red] {\small $q^4$} (3.55, 0.39);
					
		\draw[thick] (0,2) -- node[above, text=red] {\small $q$} (-1.73,3);
		
			\draw[thick] (-1.73, 3) -- node[right, text=red] {\small $q^2$} (-1.99, 4.48);	
		
				\draw[thick] (-1.99, 4.48) -- node[right, text=red] {\small $q^3$} (-1.31, 5.53);
		
					\draw[thick] (-1.31, 5.53) -- node[below, text=red] {\small $q^4$} (-0.36, 5.87);
					\draw[thick] (-1.31, 5.53) -- node[right, text=red] {\small $q$} (-1.23, 6.53);
		
				\draw[thick] (-1.99, 4.48) -- node[above, text=red] {\small $q$} (-2.98, 5.24);	
		
					\draw[thick] (-2.98, 5.24) -- node[right, text=red] {\small $q^2$} (-3.24, 6.21);
					\draw[thick] (-2.98, 5.24) -- node[above, text=red] {\small $q$} (-3.98, 5.24);
		
			\draw[thick] (-1.73, 3) -- node[above, text=red] {\small $q$} (-3.14, 2.49);
		
				\draw[thick] (-3.14, 2.49) -- node[above, text=red] {\small $q^2$} (-4.29, 2.97);
		
					\draw[thick] (-4.29, 2.97) -- node[right, text=red] {\small $q^3$} (-4.79, 3.84);
					\draw[thick] (-4.29, 2.97) -- node[above, text=red] {\small $q$} (-5.26, 2.71);
		
				\draw[thick] (-3.14, 2.49) -- node[left, text=red] {\small $q$} (-3.72, 1.38);
		
					\draw[thick] (-3.72, 1.38) -- node[above, text=red] {\small $q^2$} (-4.63, 0.96);
					\draw[thick] (-3.72, 1.38) -- node[left, text=red] {\small $q$} (-3.55, 0.39);
	\end{tikzpicture}

\end{center}
\caption{Local rules and the quantized Conway-Farey Topograph} \label{fig:QFCT}
\end{figure}
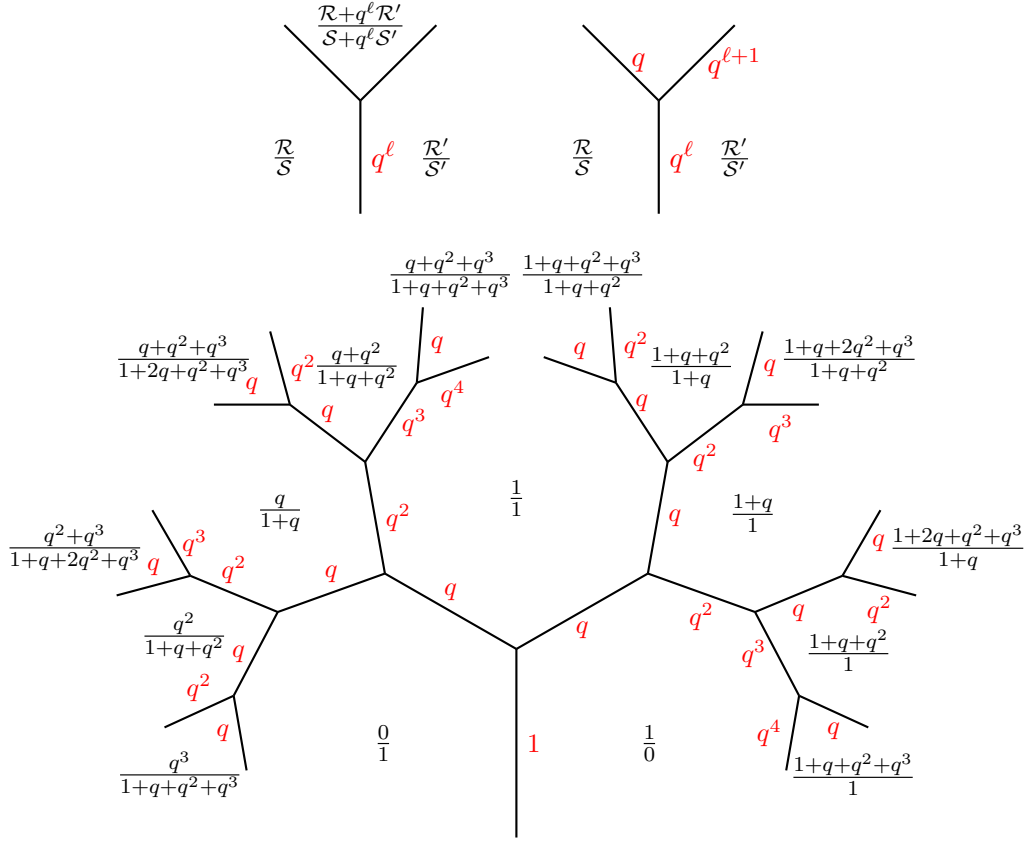


To calculate the polynomials $\mathcal{R}$ and $\mathcal{S}$ one can use the following notion of continuants. 

For $\frac{r}{s} = [a_1, a_2, \dots, a_{2m}]$ the {\it continuant} $K_{2m}^+$ is defined as the determinant
\begin{equation*}
    K_{2m}^+(a_1, \dots, a_{2m}) := \det\begin{pmatrix}
    a_1 & 1 & & & \\
    -1 & a_2 & 1 & & \\
    & \ddots & \ddots & \ddots & \\
    & & -1 & a_{2m-1} & 1 \\
    & & & -1 & a_{2m}
    \end{pmatrix}.
\end{equation*}
One can check that  
\begin{equation}
\label{rcont}
\begin{split}
    r &= K_{2m}^+(a_1, a_2, \dots, a_{2m}) \\
    s &= K_{2m-1}^+(a_2, a_3, \dots, a_{2m})
\end{split}
\end{equation}
Morier-Genoud and Ovsienko defined the $q$-continuants, with a distinction between the number of variables, as follows:
\begin{itemize}
\item For an even number of variables:
\begin{equation}
\label{evencont}
    K_{2m}^+(a_1, \dots, a_{2m})_q := \det{\begin{pmatrix}
    [a_1]_q & q^{a_1} & & & & & \\
    -1 & [a_2]_{q^{-1}} & q^{-a_2} & & & & \\
    & -1 & [a_3]_q & q^{a_3} & & & \\
    & & -1 & [a_4]_{q^{-1}} & q^{-a_4} & & \\
    & & & \ddots & \ddots & \ddots & \\
    & & & & -1 & [a_{2m-1}]_q & q^{a_{2m-1}} \\
    & & & & & -1 & [a_{2m}]_{q^{-1}}
    \end{pmatrix}}
\end{equation}
\item For an odd number of variables:
\begin{equation}
\label{oddcont}
    K_{2m-1}^+(a_2, \dots, a_{2m})_q := \det{\begin{pmatrix}
    [a_2]_{q^{-1}} & q^{-a_2} & & & & \\
    -1 & [a_3]_q & q^{a_3} & & & \\
    & -1 & [a_4]_{q^{-1}} & q^{-a_4} & & \\
    & & \ddots & \ddots & \ddots & \\
    & & & -1 & [a_{2m-1}]_q & q^{a_{2m-1}} \\
    & & & & -1 & [a_{2m}]_{q^{-1}}
    \end{pmatrix}}
\end{equation}
\end{itemize}

With this we have the following quantised versions of equations \eqref{rcont}
\begin{align*}
    \mathcal{R}(q) &= q^t K_{2m}^+(a_1, a_2, \dots, a_{2m})_q \\
    \mathcal{S}(q) &= q^t K_{2m-1}^+(a_2, a_3, \dots, a_{2m})_q,
\end{align*}
where $t = a_2 + a_4 + \dots + a_{2m} - 1$.

By expansion along the final row of these matrices (in both odd and even cases), we can define the continuants recursively by
\begin{equation}
\label{qcontrec}
    K_n^+(a_1, \dots, a_n)_q = K_{n-1}^+(a_1, \dots, a_{n-1})_{q^{-1}}[a_n]_{q^{-1}} + q^{a_{n-1}} K_{n-2}^+(a_1, \dots, a_{n-2})_q,
\end{equation}
with initial conditions $K_0^+()_q = 1, K_1^+(a_1)_q = [a_1]_q$. 

Extending the recursive formula \eqref{qcontrec}, we can determine the denominator of a $q$-rational $x\in(0,1)$ recursively as follows. For $x = [0, a_2, \dots, a_{2m+2}]$, with $[x]_q = \frac{\mathcal{R}(q)}{\mathcal{S}(q)}$ define $S_{2m+2} := \mathcal{S}(q)$. Then we have
\begin{multline}
\label{SQrecursion}
    S_{2m+2} = S_{2m} \left(\frac{[a_{2m+2}]_q}{ [a_{2m}]_q} + q 
[a_{2m+2}]_q [a_{2m+1}]_q + q^{a_{2m+1} + a_{2m+2}}\right)\\
  - S_{2m-2}\left(\frac{q^{a_{2m-1} + a_{2m}} [a_{2m+2}]_q}{[a_{2m}]_q}\right),
\end{multline}
with the initial conditions
\begin{equation}
\label{SQinitial}
    S_2 = [a_2]_q, \qquad S_4 = q[a_2]_q[a_3]_q[a_4]_q + q^{a_3 + a_4}
[a_2]_q + [a_4]_q.
\end{equation}

We have also the following general formulas, see \cite{LMG}:
\begin{itemize}
    \item Shift formula, for $n \in \mathbb{N}$,
        \begin{equation}
        \label{shift}
            [x + n]_q = q^n [x]_q + [n]_q.
        \end{equation}
    \item Negation formula,
        \begin{equation}
        \label{negation}
            [-x]_q = -q^{-1} [x]_{q^{-1}}.
        \end{equation}
    \item Inversion formula,
        \begin{equation}
        \label{inversion}
            \left[\frac{1}{x}\right]_q = \frac{1}{[x]_{q^{-1}}}.
        \end{equation}
\end{itemize}

Alternatively, for a rational number $\frac{r}{s}$, we can consider its \textit{Hirzebruch (negative) continued fraction} by
\begin{equation*}
        \frac{r}{s} = \llbracket c_1, c_2, c_3, \ldots, c_k \rrbracket
 \coloneq c_1 - \cfrac{1}{c_2 - \cfrac{1}{c_3 - \cfrac{1}{\ddots - \cfrac{1}{c_k}}}},
\end{equation*}
where $c_1$ is an integer and $c_i \geq 2$ is a positive integer for all $i \geq 2$.

\begin{definition} (Morier-Genoud, Ovsienko \cite{MGO1})  \label{def:2}
    Given a Hirzebruch continued fraction $[c_1, c_2, \dots, c_k]$, define 
its \textit{$q$-deformation} by 
    \begin{equation*}
        \llbracket c_1, c_2, \dots, c_k \rrbracket_q \coloneq [c_1]_q - \cfrac{q^{c_1 - 1}}{[c_2]_q - \cfrac{q^{c_2 - 1}}{\ddots - \cfrac{q^{c_{k-1} - 1}}{[c_k]_q}}}
    \end{equation*}
\end{definition}
It is proved in \cite{MGO1} that definitions \ref{def:1} and \ref{def:2} are
compatible, in the sense that if $[a_1,\ldots,a_{2m}]$ and 
$\llbracket c_1,\ldots,c_n\rrbracket$ represent the same rational number
then
\begin{displaymath}
  [a_1,\ldots,a_{2m}]_q = \llbracket c_1,\ldots,c_n\rrbracket_q
\end{displaymath}
as rational functions of $q$.

We can go between regular and Hirzebruch continued fractions using the following correspondence 
\begin{equation}
\label{CF Conversion}
[a_1 , a_2, \dots, a_{2m}] =
\llbracket a_1+1, 2, \dots, 2, a_3 + 2, 2, \ldots, 2,
 \cdots, a_{2m-1} + 2, 2, \dots, 2\rrbracket,
\end{equation}
where the strings of repeating $2$'s have length $a_2 - 1, a_4 - 1, \dots, a_{2m} - 1$.
In terms of Hirzebruch partial quotients, we have 
\begin{equation*}
    \deg \mathcal{R}(q) = c_1 + c_2 + \dots c_k - k \qquad \text{and} \qquad \deg \mathcal{S}(q) = c_2 + \dots c_k - k + 1
\end{equation*}
Due to the nature of the definition of $q$-rationals using the Hirzebruch form being more straightforward (i.e., not dealing with alternating $q$ and $q^{-1}$ terms), it will often be easier to use this form in calculations.

As with the regular continued fraction there is a continuant matrix in the form
\begin{equation*}
    K_{k}(c_1, \dots, c_k) := \det{\begin{pmatrix}
    c_1 & 1 & & & \\
    1 & c_2 & 1 & & \\
    & \ddots & \ddots & \ddots & \\
    & & 1 & c_{k-1} & 1 \\
    & & & 1 & c_k
    \end{pmatrix}},
\end{equation*}
and 
\begin{equation}
\label{HIRcont}
\begin{split}
    r &= K_k(c_1, c_2, \dots, c_k) \\
    s &= K_{k-1}(c_2, c_3, \dots, c_k)
\end{split}
\end{equation}
The quantum version of $K_k$ was worked-out in \cite{MGO1}:
\begin{equation*}
    K_k(c_1, \dots, c_k)_q := \det{\begin{pmatrix}
    [c_1]_q & q^{c_1 - 1} & & & \\
    1 & [c_2]_q & q^{c_2 - 1} & & \\
    & \ddots & \ddots & \ddots & \\
    & & 1 & [c_{k-1}]_q & q^{c_{k-1} - 1} \\
    & & & 1 & [c_k]_q
    \end{pmatrix}}
\end{equation*}
and equations \eqref{HIRcont} become
\begin{equation}
\label{eq:RS_continuants}
\begin{split}
    \mathcal{R}(q) &= K_k(c_1, c_2, \dots, c_k)_q \\
    \mathcal{S}(q) &= K_{k-1}(c_2, c_3, \dots, c_k)_q.
\end{split}
\end{equation}

Now, in both regular and Hirzebruch form we can determine the rational (polynomial) instead as a product of $2 \times 2$ matrices. We will focus here on the Hirzebruch form only, however the regular continued fraction case is similar and both can be found in \cite{MGO1}.

In general, for a rational $\frac{r}{s} = \llbracket c_1, c_2, \dots, c_k
\rrbracket$ we define the matrix $M(c_1, c_2, \dots, c_k)$ as
\begin{align*}
    M(c_1, c_2, \dots, c_k) &\coloneq 
    \begin{pmatrix}
        c_1 & -1 \\
        1 & 0
    \end{pmatrix}
    \begin{pmatrix}
        c_2 & -1 \\
        1 & 0
    \end{pmatrix}
    \dots
    \begin{pmatrix}
        c_k & -1 \\
        1 & 0
    \end{pmatrix} \\
    &= \begin{pmatrix}
        r & -r_{k-1} \\
        s & -s_{k-1}
    \end{pmatrix}
\end{align*}
where $r_{k-1}, s_{k-1}$ represent the numerator and denominator of the convergent consisting of just the first $k-1$ partial quotients.

Similarly, the $q$-analogue is defined to be
\begin{align}
\nonumber
    M_q(c_1, c_2, \dots, c_k) &\coloneq 
    \begin{pmatrix}
        [c_1]_q & -q^{c_1 - 1} \\
        1 & 0
    \end{pmatrix}
    \begin{pmatrix}
        [c_2]_q & -q^{c_2 - 1} \\
        1 & 0
    \end{pmatrix}
    \dots
    \begin{pmatrix}
        [c_k]_q & -q^{c_k - 1} \\
        1 & 0
    \end{pmatrix} \\
\label{MqHir}
    &= \begin{pmatrix}
        \mathcal{R} & -q^{c_k - 1} \mathcal{R}_{k-1} \\
        \mathcal{S} & -q^{c_k - 1} \mathcal{S}_{k-1}
    \end{pmatrix}.
\end{align}

\subsection{Quantum irrationals}
Let now $x\in\R$ be irrational, and $(x_n)\subseteq\Q$ a sequence of rationals
with $x_n\to x$ as $n\to\infty$.  Morier-Genoud and Ovsienko \cite{MGO2} proved that
the sequence of quantised $([x_n]_q)$ \emph{stabilises}. i.e. more and 
more terms of the Taylor expansion in $q$ become fixed.

Therefore one can define the quantisation of $x$ by
    \begin{displaymath}
        [x]_q := \sum_{k \geq 0} \varkappa_k q^k, \qquad \text{where} \qquad \varkappa_k = \lim_{n \rightarrow \infty} \varkappa_{n, k}
    \end{displaymath}
where
    \begin{displaymath}
        [x_n]_q =: \sum_{k\geq 0} \varkappa_{n, k}q^k,
    \end{displaymath}
and the sequence of coefficients $(\varkappa_k)$ are integers,
independent of the choice of approximation sequence $(x_n)$.

For example, for the golden ratio $\varphi=\frac{1+\sqrt5}{2} = [1,1,1, \ldots]$
we have $[\varphi]_q$, satisfying the relation
$$
[\varphi]_q = [1]_q+\cfrac{q}{[1]_{q^{-1}} +
  \cfrac{q^{-1}}{[1]_q+ \cfrac{q}{[1]_{q^{-1}} + \ddots}}}
= 1 + \cfrac{q}{1+\cfrac{q^{-1}}{[\varphi]_q}}.
$$
This gives the quantum golden ratio as
$$
\left[\varphi\right]_q=\frac{q^2+q-1+\sqrt{(q^2 + 3q + 1)(q^2 - q + 1)}}{2q},
$$
or, as the series
$$
\begin{array}{rcl}
\left[\varphi\right]_q
&=&
1 + q^2 - q^3 + 2 q^4 - 4 q^5 + 8 q^6 - 17 q^7 + 37 q^8 - 82 q^9 + 185 q^{10} \\[6pt]
&&- 423 q^{11} + 978 q^{12}-2283q^{13}+ 5373q^{14}-12735q^{15}+30372q^{16}\\[6pt]
&&-72832q^{17}+175502q^{18}-424748q^{19}+1032004q^{20} \cdots
\end{array}
$$
with the sequence of coefficients in~$\left[\varphi\right]_q$ coinciding (up to the alternating sign)
with the sequence \href{https://oeis.org/A004148}{A004148} of \cite{OEIS} called the ``generalized Catalan numbers''.

The radius of convergence of this power series is governed by the
closest root of $1+3q+q^2=0$ to $0$, and is
\begin{displaymath}
  R_* = \frac{3-\sqrt5}2.
\end{displaymath}

In  \cite{LMGOV} it was conjectured that for any real $x>0$,
the radius of convergence of $[x]_q$ is at least $R_*$, which can be considered as the quantum analogue of Hurwitz' theorem, that $\varphi$ is the
most badly-approximable number. This conjecture was proved for \emph{metallic numbers}
of the form $[0,n,n,n,\ldots]$, $n\in\N$ in \cite{Ren} and in the general case in \cite{EGMS}.

\section{Maximal radius of convergence and Kronecker fractions}
\label{sec:drei}

Since the minimal radius of convergence $R(x)$ of $[x]_q$ is now known to be
$R_*$, it is natural to ask what is the maximal radius of convergence.

We should exclude from consideration positive integers $n$ since $[n]_q$ is a polynomial and hence
$R(n)=\infty$.

\begin{proposition}
  For all $x\in \R_+, $ $x\not\in \Z$, the radius of convergence $R(x) \leq 1$.
\end{proposition}
\begin{proof}
Indeed, for rational $x=r/s$ by the Fundamental Theorem of Algebra the denominator $\mathcal{S}(q)$ of $[r/s]_q$
has a complex root.  Since $\mathcal{S}(0)=1$ and $\mathcal{S}$ is monic, not all 
roots can have modulus $>1$, so $R(x) \leq 1$.

For $x\not\in\Q$, $[x]_q$ diverges at $q=1$ since there are 
infinitely many integer coefficients.
\end{proof}

We would like to describe explicitly all rational numbers $x=r/s$ having maximal radius of convergence $R(x)=1$, which means that the denominator $\mathcal{S}(q)$ in $\left[ \frac{r}s \right]_q = \frac{\mathcal{R}(q)}{\mathcal{S}(q)}$ has all zeros on the unit circle.
 We call such rationals {\it Kronecker fractions} because of the following classical result.

\begin{theorem} (Kronecker \cite{K}) If $P(q)$ is a monic polynomial with integer coefficients
with all roots of absolute value at most $1$, then $P(q)$ is a product
of cyclotomic polynomials, and/or a power of $q$. 
\end{theorem}

Recall that the $n$th \emph{cyclotomic polynomial} is defined by
  \begin{displaymath}
    \Phi_n(x) = \prod_{\substack{k=1\\\gcd(k,n)=1}}^n (x - 
\rme^{2\pi\rmi k/n}):
  \end{displaymath}
for example
$$
  \Phi_1(x) = x-1, \,\,
  \Phi_2(x) = x+1, \,\,
  \Phi_3(x) = x^2+x+1, \dots,
  \Phi_{10}(x)= x^4-x^3+x^2-x+1,\dots
$$
Note that apart from $\Phi_1$, all cyclotomic polynomial have \emph{palindromic coefficients} \cite{B}.
\begin{definition}
  A polynomial $P(q)$ is called \emph{palindromic of degree $d$} if 
  \begin{equation*}
    P(q) = q^d P\left(\frac1q\right).
  \end{equation*}
If $d$ is the degree of $P$ itself, then we say simply $P$ is
palindromic.
\end{definition}
\begin{remark}
  Allowing to specify the degree of palindromicity presents us with extra
flexibility. For example the polynomial $P(q)=q^2+q$ is not palindromic
(as a degree $2$ polynomial), but it is palindromic of degree $3$:
\begin{displaymath}
  P(q) = 0q^3 + q^2 + q + 0.
\end{displaymath}
\end{remark}

We can reduce our search for Kronecker fractions to $x\in(0,\frac12)$ due to the following:

\begin{proposition}
  If $x\in \Q\cap(0,1)$ is Kronecker then so is $[x+n]_q$ and $[1-x]_q$,
for $n\in\N$.
\end{proposition}

\begin{proof}
    Writing $[x]_q = \frac{\mathcal{R}(q)}{\mathcal{S}(q)}$ and applying the shift formula \eqref{shift} we have 
$$
        [x + n]_q = q^n \frac{\mathcal{R}(q)}{\mathcal{S}(q)} + [n]_q 
        = \frac{q^n \mathcal{R}(q) + [n]_q \mathcal{S}(q)}{\mathcal{S}(q)}.
$$
    So $[x]_q$ and $[x + n]_q$ have the same denominator, and thus if $[x]_q$ is Kronecker then the same is true for $[x + n]_q$.

    Combining the shift and negation \eqref{negation} we find that
$$
        [n - x]_q = q^n [-x]_q + [n]_q 
        = q^n (-q^{-1} [x]_{q^{-1}}) + [n]_q
        = -q^{n-1} [x]_{q^{-1}} + [n]_q.
$$
  In particular, for $n = 1$ we have
\begin{equation}
        [1 - x]_q = -[x]_{q^{-1}} + 1 \\
    \label{negationshift}
        = -\frac{\mathcal{R}(q^{-1})}{\mathcal{S}(q^{-1})} + 1.
\end{equation}

    Using the formulas for the degree of $\mathcal{R}, \mathcal{S}$  and noting that $a_1 = 0$ (since $x < 1$) we have
    $k = \deg{(\mathcal{R})} = \deg{(\mathcal{S})}.$ Define the reciprocals of the polynomials $\mathcal{R}, \mathcal{S}$ as 
    $$\mathcal{R}^*=q^k \mathcal{R}(q^{-1}), \mathcal{S}^*=q^k \mathcal{S}(q^{-1})$$ 
 to rewrite equation \eqref{negationshift} as
$$
        [1 - x]_q = -\frac{\mathcal{R}^*(q)}{\mathcal{S}^*(q)} + 1 
        = \frac{\mathcal{S}^*(q) - \mathcal{R}^*(q)}{\mathcal{S}^*(q)}.
$$
    Since the roots of $\mathcal{S}^*$ are inverse to the roots of $\mathcal{S}$, $1-x$ is also Kronecker.
\end{proof}

We have also the following interchange formula.

\begin{proposition}  \label{prop:interchange}
If $x=[0,a,b,a,b,\ldots,a], \,\, a,b \in \mathbb N$ is a Kronecker fraction, then so is $x^\dagger=[0,b,a,b,a,\ldots,b]$.
Furthermore, their quantum denominators are $[a]_q f(q)$ and $[b]_q f(q)$ respectively for some Kronecker polynomial $f(q)$.
\end{proposition}

\begin{proof}
 For the continued fraction $[0, a, b, a, b, \dots, a]$, according to equations \eqref{SQrecursion}, \eqref{SQinitial}, we have the following recursive formula
 \begin{equation*}
        S_{2m+2} = S_{2m} (1 + q [a]_q [b]_q + q^{a + b}) - q^{a + b} S_{2m-2}
    \end{equation*}
    with initial conditions
    \begin{equation*}
        S_2 = [a]_q, \qquad S_4 = [a]_q(1 + q[a]_q[b]_q + q^{a+b}).
    \end{equation*}
    In both $S_2$ and $S_4$ we have $[a]_q$ multiplied by some function of $q$ that is symmetric in $a$ and $b$. The coefficients in the recursive formula, namely $(1 + q[a]_q [b]_q + q^{a + b})$ and $(-q^{a+b})$, are also symmetric in $a$ and $b$, so by induction $S_{2m} = [a]_q f(q)$ for some function $f$, depending symmetrically on $a$ and $b$. Similarly for $x^\dagger=[0, b, a, b, a, \dots, b]$ we have $S_{2m}^\dagger = [b]_q f(q)$ with the same $f(q).$

 Since $[a]_q$ is a Kronecker polynomial, we see that $f(q)$ must be a Kronecker polynomial as well. This means that $S_{2m}^\dagger = [b]_q f(q)$ is also a Kronecker polynomial and thus $x^\dagger=[0, b, a, b, a, \dots, b]$ is  Kronecker.
\end{proof}

The following result provides a strong restriction on the fractions to be Kronecker.
We say that continued fraction of $x=\frac{r}{s}=[0,a_2,\ldots,a_{2n}]$ is palindromic if the tuple $(a_2,\ldots,a_{2n})$ is 
palindromic.

\begin{theorem} \label{thm:palindromic}
The continued fraction of every Kronecker fraction is palindromic.
\end{theorem}

\begin{proof}
First observe that if $x=r/s$ is Kronecker, then the denominator $\mathcal{S}(q)$ of $[r/s]_q$ is 
palindromic.
Indeed, from Kronecker's theorem we know that $\mathcal{S}(q)$ is a product of cyclotomic polynomials. Since $\mathcal{S}(1) = s$, the cyclotomic $\Phi_1(q)=q-1$ is not a factor, and hence $\mathcal{S}$ is a product of palindromic polynomials,
which is palindromic.

 \begin{lemma}
The quantum continuant $K_k(c_1,\ldots,c_k)_q$ is a palindromic polynomial 
if and only if $c_{i} = c_{k+1-i}$, $i=1,2,\ldots$ (i.e. the list of
coefficients $(c_1,\ldots,c_k)$ is palindromic).
\end{lemma}  
    
    
   Morier-Genoud and Ovsienko \cite{MGO1} proved the following ``mirror formula" for quantum continuants:
    \begin{equation}
    \label{mirror}
        K_k(c_1, c_2, \dots, c_k)_q = q^{c_1 + c_2 + \dots + c_k - k} K_k(c_k, \dots, c_2, c_1)_{q^{-1}}.
    \end{equation}
    Assuming palindromicity of $(c_1,\ldots,c_k)$, this becomes
    \begin{displaymath}
        K_k(c_1, c_2, \dots, c_k)_q = q^{c_1 + c_2 + \dots + c_k - k} K_k(c_1, c_2, \dots, c_k)_{q^{-1}} 
    \end{displaymath}
Since $\deg(K_k(c_1,\ldots,c_k)_q) = c_1 + c_2 + \dots + c_k - k$ (see \cite{MGO1}), this implies palindromicity of $K_k(c_1,\ldots,c_k)_q$.

    Now suppose we have that $K_k(c_1, \dots, c_k)_q$ palindromic.
   From the results of Leclere and Morier-Genoud \cite{LMG}, we know that for arbitrary $(c_1, c_2, \dots, c_k)$ the trace of the matrix $M_q(c_1, c_2, \dots, c_k)$ is palindromic. Referring to \eqref{eq:RS_continuants} and \eqref{MqHir},
 this trace can be calculated as
    \begin{equation}
    \label{Trace Equation}
        \Tr M_q(c_1, \dots, c_k) = K_k(c_1, \dots, c_k)_q - q^{c_k - 1} K_{k-2}(c_2, \dots, c_{k-1})_q
    \end{equation}
    Now, we know that the degree $K_k$ is greater than that of $K_{k-2}$ (they have degree $d := c_1 + \dots + c_k - k$ and $c_2 + \dots + c_{k-1} - k + 2$ respectively). 

    This means that $\Tr M_q$ and $K_k$ are palindromic of the same degree $d$.
$$
        \Tr M_q (q) = q^d \Tr M_q(q^{-1}), \quad
        K_k (q) = q^d K_k(q^{-1}).
$$

    Replacing $q$ by $q^{-1}$ in the trace equation~\eqref{Trace Equation}, we therefore have
$$
        \Tr M_q(q^{-1}) = K_k(q^{-1})  - q^{1 - c_k} K_{k-2}(q^{-1}) ,
        $$
        $$
        q^{1-c_k} K_{k-2}(q^{-1})  = K_k(q^{-1})  - \Tr M_q(q^{-1})  \\
        = q^{-n} \left(q^{c_k - 1} K_{k-2}(q)\right).
        $$
Thus   $q^{n - 2c_k + 2} K_{k-2}(q^{-1})  = K_{k-2}(q),
$
so $K_{k-2}$ is also palindromic, but of degree $n - 2c_k + 2$. However, from the theory of $q$-continuants we know that $K_{k-2}$ has degree $n - c_1 - c_k + 2$. Equating these two expressions for the degree we conclude that $c_1 = c_k$.

    Now that we know that $K_{k-2}(c_2, \dots, c_{k-1})$ is palindromic we can repeat this process using Eq.~\eqref{Trace Equation} again on this shortened set of $c_i$'s. At each step we will obtain one more $c_i = c_{k-i+1}$ until we have palindromicity of the tuple $(c_1, c_2, \dots, c_k)$.
%
This proves the lemma.

We now conclude the proof of the theorem.  By \eqref{eq:RS_continuants},
$\mathcal{S}(q) = K_{n-1}(c_2,\ldots,c_n)_q$, so if $\mathcal{S}(q)$
is a palindromic polynomial, the lemma proves that the list
$(c_2,\ldots,c_n)$ is palindromic.  We note that the Hirzebruch
continued fraction expansion is palindromic if and only if the regular
continued fraction expansion is palindromic, as is clear from the
identification \eqref{CF Conversion}.  Therefore the regular continued
fraction expansion is palindromic.
\end{proof}

Note that the converse to the theorem is not true. Indeed, $x=\frac8{21}=[0,2,1,1,1,2]$ has palindromic continued fraction, but the denominator of $[x]_q$
$$
  \mathcal{S}(q) = q^6 + 3q^5 + 4q^4 + 5q^5 + 4q^2 + 3q +1 \\
=(q^4+2q^3+q^2+2q+1)(q^2+q+1).
$$
has the first factor, which is irreducible but not cyclotomic, so $x$ is not Kronecker.

\begin{remark} It is important here that we consider the continued fraction expansions of even length. For example, the continued fraction of $\frac{2}{5}=[0,2,2]=[0,2,1,1]$ is not palindromic, in agreement with
$$
[2/5]_q=\frac{q^3+q^2}{q^3+2q^2+q+1}.
$$
\end{remark}

There is the following classical result about palindromic continued fractions due to Serret \cite{Ser}.

\begin{theorem}(Serret (1848))
\label{Serret}
    A rational number $r/s$ with $r < s$ has a continued fraction of the form $[0, a_2, \ldots ,a_{2m}]$ with $(a_2, \dots, a_{2m}) $ palindromic if and only if $s$ divides $r^2 - 1$.
\end{theorem}

In combination with Theorem \ref{thm:palindromic} this implies:

 \begin{corollary} \label{cor:palindromic}
For any Kronecker fraction $r/s \in (0, 1)$ the denominator $s$ divides $r^2-1.$
\end{corollary}  

Corollary \ref{cor:palindromic} appears as Theorem 3.6 in the recent
work \cite{KMRWY}, where it is proved by a different method.


These results simplify the search for Kronecker fractions. In
particular, there are 9 fractions in $(0,1)$ with palindromic
continued fraction and numerator $17$:
$$
  \frac{17}{18},\quad
  \frac{17}{24},\quad
  \frac{17}{32}, \quad
  \frac{17}{36}, \quad
  \frac{17}{48}, \quad
  \frac{17}{72}, \quad
  \frac{17}{96}, \quad
  \frac{17}{144}, \quad
  \frac{17}{288}. \quad
$$
We have checked that all fractions are Kronecker, except $\frac{17}{72}=[0,4,4,4]$.

 One can also use this to identify potential infinite families of Kronecker functions. In fact, if we consider $r$ and $s$ as polynomials in the family parameter $n$, the Serret condition looks very strong. For example, if $s=n(n+1)(n+2)$ then for $r$ we have the following possibilities:
$$n^2+3n+1, \, n^2+n-1, \, 2n^2+4n+1, $$
leading to 3 families of Kronecker fractions $K6, K7, K8$ (see the Appendix). This should provide strong restrictions on suitable denominators $s(n)$, which are very special in all known examples.

Another further corollary allows us to find all Kronecker fractions with
prime denominator:

\begin{corollary}
    A rational number $r/s \in (0, 1)$ with $s$ prime is Kronecker if and only if $r = 1$ or $r = s-1$.    
\end{corollary}

\begin{proof}
Indeed, $r^2 - 1 \equiv 0 \pmod s$ with prime $s$ holds only for $r \equiv \pm 1 \pmod s$, so $r=1$ or $s-1.$
In both cases the $q$-denominator is $[s]_q$, so they are Kronecker (see case $K1$ in the appendix table).
\end{proof}

We used a computer with the SageMath software \cite{Sage} to find all potential Kronecker functions with the denominators less than 5000. This allowed us to 
identify 13 infinite families of Kronecker fractions (each one indexed by
a parameter $n\in\N$) and a further 23 examples which do not fit to any of
our families.  These findings are summarised in the Appendix.

For the infinite families we can prove that all members are Kronecker
fractions.
For example, for the fraction $\displaystyle \frac{2n+1}{4n+4} = [0,2,n,2]$,
$n\in\N$, we have
\begin{displaymath}
  \left[ \frac{2n+1}{4n+4} \right]_q = \frac{q^2 [n]_q [2]_q + q^{n+3}}%
{[2]_q\left( q [2]_q [n]_q + q^{n+2} + 1\right)} 
= \frac{\mathcal{R}(q)}{\mathcal{S}(q)}.
\end{displaymath}
By elementary algebra
$
\mathcal{S}(q)=   q [2]_q [n]_q + q^{n+2} + 1 = [n+1]_q (1+q^2),
$
so these fractions are Kronecker (family $K2$ below). By the interchange formula (Proposition \ref{prop:interchange}) the same is true for
\begin{displaymath}
  \frac{2n+1}{2n(n+1)} = [0,n,2,n]
\end{displaymath}
 with denominator
$
 \mathcal{S}(q)= [n]_q [n+1]_q (1+q^2)
$
(family $K4$).

\section{Concluding remarks}

A natural extension would be to look also at the set of quantum
irrational numbers having radius of convergence $1.$

Morier-Genoud and Ovsienko \cite{MGO2}  suggested that the radius of convergence of
$[\pi]_q$ is $1$ due to the (experimentally observed) slow growth
of coefficients. This could mean that this property is typical for irrational numbers.
We can claim, however, that this never happens for quadratic irrationals.

\begin{proposition}
All quantum quadratic irrationals have radius of convergence strictly less than 1.
\end{proposition}

\begin{proof}
To prove this, we use a classical theorem of Fatou \cite{Fatou}:

\begin{theorem}
(Fatou (1906)) If $f(z)$ is a function whose power series has 
integer coefficients and radius of convergence equal to $1$, then
$f$ is either rational or transcendental.
\end{theorem}

Since quadratic irrationals have continued fractions that are
(eventually) periodic, their quantisation satisfies an algebraic
equation, so thus cannot have radius of
convergence $1$ by Fatou's Theorem.
\end{proof}


There is a class of {\it rational (or two-bridge) knots and links} labelled by the continued fractions (see \cite{Lickorish}).
Due to \cite{LS, MGO1} the (normalised) Jones polynomials of rational knot $K(r/s)$ can be computed in terms of the corresponding quantum rational $[r/s]_q=\frac{\mathcal R(q)}{\mathcal S(q)}$ as
$$
J_{\frac{r}{s}}(q)=q\mathcal R(q)+(1-q)\mathcal S(q).
$$

It is natural to ask what is special about the knots/links $K(r/s)$, corresponding to the Kronecker fractions $r/s$ (apart from palindromic symmetry).

Some examples of knots and links corresponding to the fractions from the Kronecker families $K1-K5$ are shown on Fig. 2 (taken from Rolfsen's book \cite{Rolfsen}).

\begin{figure}[h]
   \begin{center}
 \includegraphics[width=1.7cm]{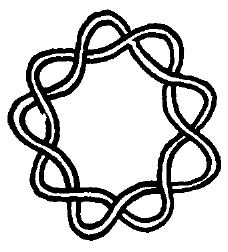} \quad  \includegraphics[width=1.8cm]{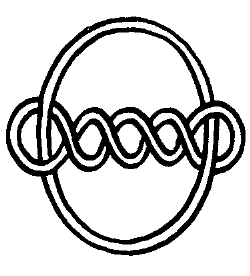}  \quad \includegraphics[width=2cm]{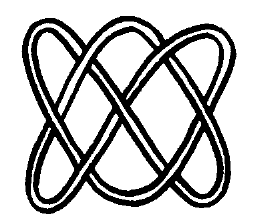} \quad \includegraphics[width=1.8cm]{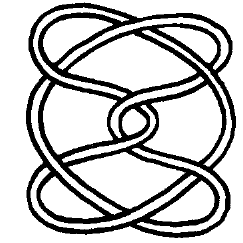} \quad \includegraphics[width=2cm]{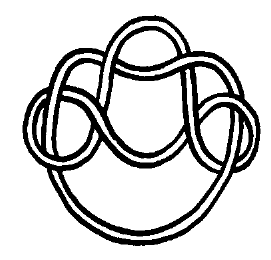}
\caption{Knots and links corresponding to the Kronecker fractions $[0,9],\, [0,2,4,2], \,[0,3,1,3], \,[0,3,2,3],\, [0,2,1,2,1,2].$}
 \end{center}
 \end{figure}

Note also that the quantum rationals play an important role for the Burau specialization problem \cite{BB}: 
{\it At which specializations of $t\in\mathbb C^*$ is the Burau representation $\rho_3$ faithful?}

Recall that the (reduced) Burau representation of the Artin's braid group $\rho_n: B_n \to GL(n-1, \mathbb Z[t,t^{-1}])$ in the case $n=3$
 is defined by
\begin{equation}
\label{rho3}
\begin{array}{rclrcl}
\rho_3\; : \quad \sigma_1 &\mapsto& 
\begin{pmatrix}
-t&1\\[4pt]
0&1
\end{pmatrix}, & \sigma_2 &\mapsto& \begin{pmatrix}
1&0\\[4pt]
t&-t
\end{pmatrix}.
\end{array}
\end{equation}

Define $\Sigma\subset \mathbb C^*$ as the union of complex poles of all $q$-rationals and $\Sigma_*:=\Sigma \cup \{1\}$. 

\begin{theorem}
(\cite{MGOV_2024}) The Burau representation $\rho_3$ specialized at $t_0\in  \mathbb C^*$ is faithful if and only if $-t_0 \notin \Sigma_*.$
\end{theorem}
 
The braids, corresponding to the Kronecker fractions, are special since they belong to the kernel of the Burau representation specialized only at roots of unity.

\section*{Acknowledgements}
We are grateful to Nick Ovenhouse and Alexey Ustinov for useful
discussions and to Prof.~K. Yanagawa for bringing our attention to
their work \cite{KMRWY}.


\appendix
\section{Known Kronecker fractions}
\label{numKro}

We present here all Kronecker fractions in the interval $(0,\frac12)$ with denominator (in reduced form) not exceeding $5000$.

\subsection{Kronecker families}
\label{pgc}
In all of the cases in the following table, we have looked at the $q$-denominators for $n = 1, 2, \dots, 10$ and used that to predict a general form that we have then been able to prove rigorously in the same manner as the examples at
the end of section \ref{sec:drei}. In some of the lengthier case, the symbolic
manipulation software Maple \cite{Maple} was used to perform the calculations.

The parameter $n$ here is assumed to be any positive integer since for $n=1$ we can use the equality
\begin{equation*}
    [0, a_1, \dots, a_{k-1}, a_k, 0, a_{k+1}, a_{k+2}, \dots, a_{n}] = [0, \dots, a_{k-1}, a_k + a_{k+1}, a_{k+2}, \dots, a_n]
\end{equation*}
 
\begin{center}
 \renewcommand{\arraystretch}{1.4}
    \begin{tabular}{|c|c|c|c|}
    \hline
        Case & Continued Fraction & Rational & $q$-Denominator \\
        \hline
        $K1$ & $[0, n+1]$ & $\frac{1}{n+1}$ & $[n+1]_q$ \\
        \hline
        
        $K2$ & $[0, 2, n-1, 2]$ & $\frac{2n-1}{4n}$ & $[2]_q [n]_q (1 + q^2)$ \\
        \hline
        
        $K3$ & $[0, n, 1, n]$ & $\frac{n+1}{n(n+2)}$ & $[n]_q [n+2]_q$ \\
        \hline
        
        $K4$ & $[0, n, 2, n]$ & $\frac{2n+1}{2n(n+1)}$ & $[n]_q [n+1]_q (1 + q^2)$ \\
        \hline
        
        $K5$ & $[0, n, 1, n, 1, n]$ & $\frac{n^2 + 3n + 1}{n(n + 1)(n + 3)}$ & $[n]_q[n+1]_q[n+3]_q$ \\
        \hline

        $K6$ & $[0, n, n+3, n]$ & $\frac{n^2 + 3n + 1}{n(n + 1)(n + 2)}$ & $[n]_q[n + 1]_q[n + 2]_q(1 - q + q^2)$ \\
        \hline
        
        $K7$ & $[0, n+2, n-1, n+2]$ & $\frac{n^2 + n -1}{n(n + 1)(n + 2)}$ & $[n]_q[n + 1]_q[n + 2]_q(1 - q + q^2)$ \\
        \hline

        $K8$ & $[0, n, 1, 2n, 1, n]$ & $\frac{2n^2 + 4n + 1}{2n(n+1)(n+2)}$ & $[n]_q [n+1]_q [n+2]_q(1+q^{n+1})$ \\
        \hline

        $K9$ & $[0, n+1, 1, n-1, 1, n+1]$ & $\frac{n^2 + 3n + 1}{n(n+2)(n+3)}$ & $[n]_q [n+2]_q [n+3]_q$ \\
        \hline
        $K10$ & $[0, 2n + 1, 1, n-1, 1, 2n + 1]$ & $\frac{2n^2 + 4n + 1}{4n(n + 1)(n + 2)}$ & $[n]_q [n+1]_q [n+2]_q (1 + q^{n+1})^2$ \\
        \hline

        $K11$ & $[0, n, 2, 2n, 2, n]$ & $\frac{8n^2 + 8n + 1}{4n(n + 1)(2n + 1)}$ & $[n]_q [n+1]_q [2n+1]_q (1 + q^2)^2$ \\
        \hline

        $K12$ & $[0, n, 1, 2n, 1, 2n, 1, n]$ & $\frac{4n^3 + 12n^2 + 9n + 1}{n(n + 2)(2n + 1)(2n + 3)}$ & $[n]_q [n+2]_q [2n+1]_q [2n+3]_q$ \\
        \hline

$K13$ & $[0, n, 1, n, 2n+2, n, 1, n]$ & $\frac{2n^4 + 8n^3 + 12n^2 + 8n + 1}{2n (n+1)^3 (n+2)}$ &$[n]_q [n+1]_q^3 [n+2]_q (1 + q^{n+1}) (1 - q + q^2)$ \\
\hline
    \end{tabular}
\end{center}

\newpage

\subsection{Exceptional (sporadic) Kronecker fractions}

In the following table we present all remaining Kronecker fractions with denominator less than 5000, which have not yet been found to be part of a family. 
We do not rule out the possibility that these fractions might belong to
some family that we have not yet identified.

Note that no new Kronecker fractions have been found with denominators
between 4000 and 5000, which might suggest that the list is not far
from complete.

\begin{center}
 \renewcommand{\arraystretch}{1.4}
    \begin{tabular}{|c|c|c|c|}
    \hline
        Case & Continued Fraction & Rational & $q$-Denominator \\
        \hline
        
        $E1$ & $[0, 2, 1, 1, 2, 1, 1, 2]$ & $\frac{31}{80}$ & $[2]_q^3[5]_q(1+q^2)$ \\
        \hline
    
        $E2$ & $[0, 3, 3, 1, 3, 3]$ & $\frac{49}{160}$ & $[2]_q^2[5]_q(1+q^2)^3$ \\
        \hline

        $E3$ & $[0, 3, 2, 1, 1, 1, 2, 3]$ & $\frac{71}{240}$ & $[2]_q^2[3]_q[5]_q(1+q^2)^2$ \\
        \hline

        $E4$ & $[0, 2, 1, 2, 3, 2, 1, 2]$ & $\frac{89}{240}$ & $[2]_q^2[3]_q[5]_q(1+q^2)^2$ \\
        \hline

        $E5$ & $[0, 2, 3, 1, 2, 1, 3, 2]$ & $\frac{127}{288}$ & $[2]_q^3[3]_q^2(1+q^2)^2(1-q+q^2)$ \\
        \hline

        $E6$ & $[0, 2, 2, 1, 5, 1, 2, 2]$ & $\frac{134}{315}$ & $[3]_q^2[5]_q[7]_q$ \\
        \hline

        $E7$ & $[0, 3, 1, 1, 6, 1, 1, 3]$ & $\frac{99}{350}$ & $[2]_q[5]_q^2[7]_q$ \\
        \hline


        $E8$ & $[0, 2, 3, 2, 1, 2, 3, 2]$ & $\frac{209}{480}$ & $[2]_q^2[3]_q[5]_q(1 + q^2)^3$ \\
        \hline

        $E9$ & $[0, 3, 2, 1, 4, 1, 2, 3]$ & $\frac{161}{540}$ & $[2]_q[3]_q^3[5]_q(1 + q^2)(1-q+q^2)$ \\
        \hline

        $E10$ & $[0, 4, 1, 1, 6, 1, 1, 4]$ & $\frac{127}{576}$ & $[2]_q^5[3]_q^2(1 + q^2)(1-q+q^2)^3$ \\
        \hline

        $E11$ & $[0, 2, 2, 1, 1, 3, 1, 1, 2, 2]$ & $\frac{251}{600}$ & $[2]_q^2[3]_q[5]_q^2(1 + q^2)$ \\
        \hline

        $E12$ & $[0, 2, 3, 1, 1, 2, 1, 1, 3, 2]$ & $\frac{351}{800}$ & $[2]_q^3[5]_q^2(1 + q^2)^2$ \\
        \hline

        $E13$ & $[0, 2, 6, 1, 2, 1, 6, 2]$ & $\frac{391}{840}$ & $[2]_q[3]_q[5]_q[7]_q(1 + q^2)(1 + q^4)$ \\
        \hline

        $E14$ & $[0, 3, 1, 1, 2, 2, 2, 1, 1, 3]$ & $\frac{251}{900}$ & $[2]_q[3]_q^2[5]_q^2(1 + q^2)$ \\
       \hline
        
        $E15$ & $[0, 2, 7, 4, 7, 2]$ & $\frac{449}{960}$ & $[2]_q^3[3]_q[5]_q(1 + q^2)^2(1 - q + q^2)^2(1 + q^4)$ \\
        \hline
 
  $E16$   & $[0, 2, 1, 1, 2, 1, 3, 1, 2, 1, 1, 2]$ & $\frac{559}{1440}$ & $[2]_q^4[3]_q^2[5]_q(1 + q^2)(1 - q + q^2)$ \\
        \hline

 $E17$    & $[0, 4, 1, 2, 7, 2, 1, 4]$ & $\frac{323}{1512}$ & $[2]_q^2[3]_q^3[7]_q(1 + q^2)(1 - q + q^2)^2$ \\
        \hline

 $E18$    & $[0, 2, 6, 2, 1, 2, 6, 2]$ & $\frac{701}{1512}$ & $[2]_q^2[3]_q^3[7]_q(1 + q^2)(1 - q + q^2)^2$ \\
        \hline

 $E19$    & $[0, 3, 1, 1, 3, 2, 3, 1, 1, 3]$ & $\frac{449}{1600}$ & $[2]_q^3[5]_q^2(1 + q^2)^3$ \\
        \hline

 $E20$    & $[0, 2, 2, 1, 5, 1, 5, 1, 2, 2]$ & $\frac{919}{2160}$ & $[2]_q^2[3]_q^3[5]_q(1 + q^2)(1 - q + q^2)(1 + q^4)$ \\
        \hline

$E21$    & $[0, 2, 4, 1, 3, 1, 3, 1, 4, 2]$ & $\frac{1217}{2688}$ & $[2]_q^4[3]_q[7]_q(1 + q^2)^3(1 - q + q^2)$ \\
        \hline

$E22$    & $[0, 3, 7, 1, 4, 1, 7, 3]$ & $\frac{1151}{3600}$ & $[2]_q[3]_q[5]_q^2(1 + q^2)^2(1 + q^4)(1 + q^3 + q^6)$ \\
        \hline

$E23$    & $[0, 2, 1, 2, 1, 1, 1, 3, 1, 1, 1, 2, 1, 2] $ & $\frac{1409}{3840}$ & $[2]_q^6[3]_q[5]_q(1 + q^2)^2(1 - q + q^2)$ \\
        \hline
    \end{tabular}
\end{center}

 \end{document}